\documentclass[12pt]{amsart}

\usepackage{amsfonts,amsmath,amsthm, amssymb,amscd}

\usepackage{graphicx}

\newcommand{\ignore}[1]{ }

\newtheorem{defn}{Definition}
\newtheorem{lemma}{Lemma}
\newtheorem{theorem}{Theorem}
\newtheorem{prop}{Proposition}

\newtheorem{conjecture}{Conjecture}

\newtheorem{cor}{Corollary}

\def\dtv{d_{\rm TV}}
\def\e{\mathbb{E \,}}
\def\var{{\rm Var \,}}
\def\p{\mathbb{P}}
\def\R{\mathbb{R}}

\def\bp{{\bf p}}
\def\bq{{\bf q}}
\def\bpnx{{\bp (n,x)}}
\def\bpnxn{{\bp (n,x_n)}}
\def\DD{D(\bp)}
\def\DDLR{D(\bp,\bq)}
\def\DDLL{D(\bp,\bp)}
\def\psn{\bp^{(n)}}
\def\qsn{\bq^{(n)}}

\def\DDn{D(\psn)}

\def\PRZ{\mathbb{P}_{\mathbb{Z}_+}}
\def\Zp{\mathbb{Z}_+}

\title[Random Sampling of Pairs]{On the Random Sampling of Pairs, with %Pedagogical 
Pedestrian
examples}

\author[Arratia]{Richard Arratia}
\address[Richard Arratia]{Department of Mathematics, University of Southern California,
Los Angeles CA 90089.}
\email{rarratia@math.usc.edu}

\author[DeSalvo]{Stephen DeSalvo}
\address[ Stephen DeSalvo]{Department of Mathematics, UCLA
Los Angeles CA 90095.}
\email{stephendesalvo@math.ucla.edu}

\date{June 1, 2013}

\begin{document}

\ignore{
\begin{abstract}
Suppose one desires to randomly sample a pair of objects such as socks, hoping to get a matching pair.
One approach is 
to sample two at a time, over and over without memory, until a matching pair is found.  A second approach is to sample sequentially, one at a time, with memory, until the same color has been seen twice.

We study the 
difference between these two methods.  The input is a discrete probability distribution on
colors, describing what happens when one sock is sampled.  
The output, a number we call the \emph{discrepancy}
of the input distribution, is the total variation distance between the two derived distributions.
 
It is easy to determine when the two pair-color distributions come out equal; that is, 
which distributions have discrepancy zero, but hard to determine the largest possible discrepancy.  
We give a plausible conjecture
for the general situation of a finite number of colors, and give an exact computation of a constant which is
a plausible candidate for the supremum of the discrepancy over all discrete probability distributions. 

We briefly consider the more difficult case where the objects to be matched into pairs are of two different kinds, such as male-female  or  left-right.
\end{abstract}
}
%\ignore{
\begin{abstract}
Suppose one desires to randomly sample a pair of objects such as socks, hoping to get a matching pair.
Even in the simplest situation for sampling, which is sampling \emph{with} replacement, 
the innocent phrase ``the distribution of the color of a matching pair'' is ambiguous.
One interpretation is that we condition on the event of getting a match between two random socks;  this corresponds to sampling two at a time, over and over without memory, until a matching pair is found.  A second interpretation is to sample sequentially, one at a time, with memory, until the same color has been seen twice.

We study the 
difference between these two methods.  The input is a discrete probability distribution on
colors, describing what happens when one sock is sampled.  There are two derived distributions --- the pair-color distributions under the two methods of getting a match.  The output, a number we call the \emph{discrepancy}
of the input distribution, is the total variation distance between the two derived distributions.
 
It is easy to determine when the two pair-color distributions come out equal, that is, 
to determine
which distributions have discrepancy zero, but hard to determine the largest possible discrepancy.  
We find the exact extreme for the case of two colors, by analyzing the roots of a fifth degree polynomial in one variable.
We find the exact extreme for the case of three colors, by analyzing the 49 roots of a variety spanned by two seventh-degree polynomials in two variables.  We give a plausible conjecture
for the general situation of a finite number of colors, and give an exact computation of a constant which is
a plausible candidate for the supremum of the discrepancy over all discrete probability distributions. 

We briefly consider the more difficult case where the objects to be matched into pairs are of two different kinds, such as male-female  or  left-right.
\end{abstract}
%}

\maketitle

\tableofcontents

%\section{Introduction, with a simple example}
\section{Motivation}

\ignore{
%The problem that inspires us is:  Supposing a drawer has 6 white and 2 black socks, or 20 white and 60 black,
%or 2000 white and 6000 black; 
%how many socks must one remove to ensure a pair of matching color?   
%The problem that we address here is different:  what is the distribution of the color of a matching pair?

The problem that inspires us:  Suppose a drawer with black and white socks and has at least one pair of matching colors.
% in proportion $1:3$; that is, there are three times as many white socks as black socks.
How many socks must one remove to ensure a pair of matching color?   The answer, 3, follows from the pigeon-hole principle, and leads us to the problem that we address in this paper: what is the distribution of the color of a matching pair?
%the problem that we address here is different:  what is the distribution of the color of a matching pair?
}

The problem that inspires us:  Suppose a drawer  has 12 white and 4 black socks.
How many socks must one remove to ensure a pair of matching color?   
The answer, 3, illustrates %follows from 
the pigeon-hole principle.
The statement of detailed counts, 12 and 4, was arbitrary, %was thrown in as a red herring, %which naturally
%, and 
but
leads  to the problem that we address in this paper: what is the distribution of the color of a matching pair?

To simplify,  we take the limit as the number of socks in the drawer goes to infinity, while the proportions remain
constant, e.g., seventy five percent white and twenty five  percent  black.

% DeSalvo Edit 8-8-12 ------
%But notice, there are %still 
%two sensible methods for choosing ``a matching pair."  
% replaced with   -------
We consider two sensible methods for choosing ``a matching pair."
% --------

\begin{itemize}
\item[(M1)] Select objects two at a time until a pair of the same color is selected in a single round;
\item[(M2)] Select objects one at a time  until the first pair of the same color is found.
\end{itemize}

%Write $N$ for the number of socks inspected, in either method.
%Method 2 corresponds to the pigeon-hole principle;  in method 1 the number of socks to inspect is unbounded.
For a second example, if there are 365 equally likely colors for socks, then, under Method 2  the maximum number of socks inspected is 366, but the expected number is $23.6166 \ldots$ .\footnote{The exact computation is $\e N =\sum_{k \ge 0} \p(N>k) = \sum_{k=0}^{364} (365)_k/365^k$, with the notation $(n)_k = n!/(n-k)!$ for $n$
falling $k$.}
%Note, this exact value is not the familiar 23, which is very nearly the median value.  
In contrast, the expected number of pairs inspected by Method 1 is exactly 365,  hence the expected number of socks inspected is 730.
However, our focus is not on the \emph{number} of socks inspected, but rather, on the distribution of the \emph{color} of the matching pair.
\ignore{
A longer version of the longer winded discussion of N would have, for method 2,  P(N>k) = P(D_k)  where D_k is the event that k people have distinct birthdays, with  P(D_k)  = (365)_k / 365^k, so  EN = sum_{k \ge 0} P(N>k) = sum_{k=0}^{365}  (365)_k / 365^k = (exact value).  And with f_2 = 1/365, Method 1 involves the geometric distribution,  with  P(N/2 \ge k) = (1-f_2)^k  so that E N = 2/f_2.}

In our first example, under method 1 the odds for a white pair over a black pair are $(12/16)^2$ to $(4/16)^2$; equivalently
$12^2$ to $4^2$, or $3^2$ to $1^2$, so that 9/10 of the time the pair is white, and 1/10 of the time it is black.  Under
method 2, the outcomes resulting in a white pair correspond to $ww,bww,wbw$, with total probability
$(.75)^2 + 2 (.75)^2 (.25)^2 = 27/32$, and the outcomes resulting in a black pair correspond to $bb,wbb,bwb$, with 
total probability
$(.25)^2 + 2 (.75) (.25)^2 = 5/32$.

To summarize, the input is a distribution on colors, $\bp=(.75,.25)$, and there are two outputs:  under Method 1, the color of a pair is white with probability $.9$, and black with probability $.1$, while under Method 2, color of a pair is white with probability $27/32$, and black with probability $5/32$.
\begin{eqnarray*}
 \bp & = & (.75,.25) \\
 M1(\bp) &=& (.9,.1)  \\ 
 M2(\bp)  & =& ( .84375,.15625) .
\end{eqnarray*}

Some natural questions, for an arbitrary 
discrete   % added NOv 8, 2012 
distribution $\bp$ for the color of a single sock:
\begin{eqnarray*}
& \mbox{(Q1)} & \mbox{When does $M1(\bp)=M2(\bp)$?} \\
&  \mbox{(Q2)}& \mbox{ How far apart can $M1(\bp)$ and $M2(\bp)$ be from each other?}
\end{eqnarray*}

There are practical algorithms \cite{PDC} for sampling, exploiting the birthday paradox, that require getting a matching pair whose color has  the distribution (M1), but under a naive \emph{opportunistic} implementation, would only find a pair whose color is distributed according to (M2).   
Question (Q2) above is about quantifying the error that would result from using the opportunistic implementation.

\ignore{
Clearly these two methods get a different distribution on the color of the matching pair.   A standard and natural way to
quantify the difference is the \emph{total variation distance}, see Section \ref{sect dtv}, which for our simple example has the value
$$
  \left| \frac{9}{10} - \frac{27}{32} \right| = 
  \frac{1}{2} \left( \left| \frac{9}{10} - \frac{27}{32} \right| + \left| \frac{1}{10} - \frac{5}{32} \right| \right).
$$

To summarize, the input was the distribution on colors of single socks, $\bp = (3/4,1/4)$, and the output 
$D(\bp)$ is the total variation distance between the distributions on colors of a pair coming from the two natural methods of picking a pair, $D(\bp) =9/10-27/32=.05625$.   
}  % end ignore 

%\section{Introduction, with the general case}
\section{Pair-derived distributions}

In general, we write $S$ for the random color of a single sock, and describe the initial distribution of colors
with 
$$  p_i := \p(S=i).$$
When the number of colors is finite, say $n+1$, then we let the colors be $0,1,2,\ldots,n$, and the distribution of $S$ is given by  $\bp=(p_0,p_1,\ldots,p_n)$.  Our initial example had $n+1=2$, $\bp=(p_0,p_1)=(.75,.25)$.
When the number of colors is infinite, we take the colors to be $0,1,2,\ldots$, and then
$\bp=(p_0,p_1,p_2,\ldots)$.    
%Without loss of generality we may assign color numbers in order of popularity, that is, 
%$$  p_0 \ge p_1 \ge \cdots  \ge 0.$$

Method 1 may be described as the color $X$ of a pair of randomly chosen socks, conditional on getting a match.  More precisely, the two chosen socks have colors $S$ and $S'$ and are independent and identically distributed, with $\p_i = \p(S=i)$.
We write 
\begin{equation}\label{def f2}
f_2 := \p(S=S')  = \sum_i \p(S=S'=i) = \sum_{i} p_i^2
\end{equation}
for the probability that two randomly chosen socks match,
so
\begin{equation}\label{def M1}
  \p(X=i) = \p(S=i|S=S') = \frac{p_i^2}{f_2}.
\end{equation}

Method 2 involves a sequential procedure:  pick socks one at a time until a duplicate color is found.  Suppose that when this duplicate is found, there have been $k$ \emph{other} colors, with $k=0,1,2,\ldots$.  Write
$i$ for the duplicate color, and $J=\{j_1,\ldots,j_k\}$ for the single colors, so that $i \notin J$ and $|J|=k$.
The second occurrence of color $i$ is at time $k+2$, and for the first $k+1$ socks, any permutation of the colors in
$\{i\} \cup J$ is valid.  Hence the color $Y$ of the matching pair found by Method 2 has distribution given by
\begin{equation}\label{def M2}
  \p(Y=i) =  p_i^2  \sum_{k} (k\! + \! 1)! \,  \sum_J    p_{j_1} \ldots p_{j_{k}}.
\end{equation}
In the sum above, $|J|=k$ and $i \notin J$.

% end ignore
 
\section{When are the two pair-picking methods the same?} 

 A discrete distribution is said to be \emph{uniform} if it has finite support, say of size $n+1$, and for each color $i$ in the support, $p_i=1/(n+1)$.   %It is easy to see that
 The following proposition is trivial.\footnote{
Because, in fact, if $\bp$ is a  uniform distribution, then  both $M1(\bp)$ and $M2(\bp)$ are equal to the original uniform distribution --- by the principle of ignorance, all possible colors are alike, and hence, equally likely under each of the derived methods.  We invite the reader to consider, is ``principle of ignorance,'' i.e. invoking symmetry, without presenting details as in \eqref{equal}, an adequate proof?}
 \begin{prop}
 if $\bp$ is uniform, then $M1(\bp)  = M2(\bp)$.
 \end{prop}
The converse is true, but not so easy to prove;  we will first prove an ancillary result in Lemma \ref{ratio lemma} 
%place the brunt of the proof in Lemma \ref{ratio lemma},
and then summarize %the result 
in Theorem \ref{uniform thm}.

\begin{lemma}\label{ratio lemma}
Under Method 2, as specified by \eqref{def M2},
\begin{equation}\label{ratio nonstrict}
\mbox{if } p_i \ge p_j >0 , \mbox{ then } \frac{\p(Y=i)}{p_i^2} \leq \frac{\p(Y=j)}{p_j^2},
\end{equation}
hence
\begin{equation}\label{equal}
\mbox{if } p_i=p_j >0 \mbox{ then } \p(Y=i)=\p(Y=j).
\end{equation}
Also,
\begin{equation}\label{eq ratio}
\mbox{if } p_i > p_j >0 , \mbox{ then } \frac{\p(Y=i)}{p_i^2} < \frac{\p(Y=j)}{p_j^2}.
\end{equation}
\end{lemma}
\begin{proof}
Assume $p_i \ge p_j >0$. Define $t(i,k)$ to be the inner sum of \eqref{def M2}, so that 
$$
\frac{\p(Y=i)}{p_i^2} = \sum_k (k\!+\!1)!  \ t(i,k).
$$ 
 To prove \eqref{ratio nonstrict} it suffices to show that if $p_i \ge p_j >0$ then $t(i,k) \leq t(j,k)$ for all $k$, and to further prove 
 \eqref{eq ratio}, it suffices to show that if $p_i >p_j$ then $t(i,k) < t(j,k)$    for at least one $k$.  With sums always taken over sets of size $k$, 
$$
t(i,k) = \sum_{i \notin J} p_{i_1}\cdots p_{i_k}  =  \sum_{i \notin J, j \in J} p_{i_1}\cdots p_{i_k}  +  \sum_{i,j \notin J} p_{i_1}\cdots p_{i_k} ,
$$
that is, in the sum over sets $J$ excluding $i$, we take cases according to whether or not $j \in J$.
With a similar decomposition of $t(j,k)$, taking  the difference
 yields
% From AMM referee's suggestion
$$
t(i,k)-t(j,k)=k(p_j - p_i) \sum_{i,j\notin J} p_{i_1}\cdots p_{i_{k-1}}.
%\t(i,k) - t(j,k) = k (p_i - p_j) 
$$
% Original finish of proof
\ignore{ $$
t(i,k)-t(j,k)  =   \sum_{i \notin J, j \in J} p_{i_1}\cdots p_{i_k}    - \sum_{i \in J, j \notin J} p_{i_1}\cdots p_{i_k}.
$$
There is a bijection between sets $J$ for the first sum and sets $J$ for the second sum, that 
   substitutes $i$ for $j$. 
%   Since $p_i > p_j$, it then follows that $t(i,k) < t(j,k)$ for all $k>0$, and $t(i,0)=t(j,0)$.
From $p_i \ge p_j$ it follows that for all $k$, $t(i,k) \le t(j,k)$, and further, when $p_i > p_j$, we have $t(i,1) < t(j,1)$.}
\end{proof}

\begin{theorem}\label{uniform thm}
Over all discrete distributions $\bp$, the derived distributions of $X$ and $Y$, given by \eqref{def M1} and
\eqref{def M2}, are equal if and only if $\bp$ is a uniform distribution.
\end{theorem}
\begin{proof}

% AS THE REVIEWER COMMENTED, p uniform implies L(X) = L(Y) is already stated as easily seen in the intro paragraph of this section.
%Assume first that $\bp$ is a uniform distribution, say over $n+1$ colors, so that for all $i,j$ in the support of $\bp$, we have $p_i=p_j = 1/(n+1) $.   For $i,j$ both in the support of $\bp$, it is obvious from \eqref{def M1} that $p_i=p_j$ implies $\p(X=i)=\p(X=j)$, and \eqref{equal} shows that  
%$\p(Y=i)=\p(Y=j)$.  Hence for $i$ in the support of $\bp$,  $\p(X=i)=1/(n+1)=\p(Y=i)$, implying that $X$ and $Y$ have the same distribution.
 
%To prove the opposite direction, s
Suppose $\bp$ is \emph{not} a uniform distribution.  Then we can fix $i,j$ with
$p_i > p_j > 0$.  From \eqref{eq ratio}, we get
$$
\frac{\p(Y=i)}{p_i^2} < \frac{\p(Y=j)}{p_j^2}, 
$$
and dividing by $f_2$ to relate with \eqref{def M1}, and rearranging, 
\begin{equation}\label{ratio XY}
\frac{\p(X=i)}{\p(X=j)} > \frac{\p(Y=i)}{\p(Y=j)},
\end{equation}
which implies that $X$ and $Y$ have different distributions.  
\end{proof}

Theorem \ref{uniform thm} gives a complete answer to our first question:  when are the two pair-picking methods the same?  Next we turn to the second question:  when the two methods are different, how different can they be?

\section{Total variation distance}
\label{sect dtv}
%  DeSalvo 8-8-12
%We want to quantify: given a probability distribution $\bp$, how far apart are the two distributions on the color of a pair of socks, with the matching pair chosen by Method 1 or Method 2.  
% replaced with
We wish to quantify: given a probability distribution $\bp$, with the matching pair chosen by Method 1 or Method 2, how far apart are the two distributions with respect to the color of the matching pair?
% ---------  Reason is because this is a question, and the phrasing before added a statement to the end of a question so the question mark looked funny.  Now the inflection going up at the end matches the question mark.

A metric on the space of all probability measures is the
\emph{total variation distance}.  
\begin{defn}\label{define dtv}
For two real-valued random variables $X$ and $Y$, the total variation distance between the \emph{laws} of $X$ and $Y$ is defined as
$$\dtv(\mathcal{L}(X),\mathcal{L}(Y)) = \sup_{A\subseteq \R} |P(X\in A) - P(Y\in A)|,$$
where the sup is taken over all Borel sets $A\subseteq \R$.  When there is no confusion, we write $\dtv(X,Y)$ instead of $\dtv(\mathcal{L}(X),\mathcal{L}(Y)).$
\end{defn}
%It is common to write $\dtv(X,Y)$ instead of $\dtv(\mathcal{L}(X),\mathcal{L}(Y)).$     
This choice of definition is useful for
probability, with the desirable property that  $\dtv(X,Y) \le 1$, and it equals $\sup_{f: \mathbb{R} \to [0,1]} | \e f(X) - \e f(Y)|$.\footnote{But there is an alternate tradition, from analysis, to define the total variation distance between measures $\mu,\nu$ as $\sup_{f: \mathbb{R} \to [-1,1]} | 
 \int f d\mu  - \int f d\nu |$, which, when applied to $\mu = \mathcal{L}(X), \nu = \mathcal{L}(Y)$, gives values ranging from 0 to 2.}
 
%Some elementary facts about total variation distance:  
  % Note also $0\leq \dtv(X,Y) \leq 1$.  
 When $X$ and $Y$ are discrete random variables,  an equivalent definition is 
\begin{equation}\label{def dtv}
\dtv(X,Y) = \frac{1}{2} \sum_k |\p(X=k)-\p(Y=k)|.
\end{equation}
%Note that without the absolute value , the sum is identically 0.  
Furthermore, since 
$\sum_k \p(X=k) = \sum_k \p(Y=k)$, we can divide the summands into positive and negative parts to obtain two more equivalent definitions.\footnote{Notation:  $t^+=\max(0,t),t^-=\max(0,-t)$; hence $|t|=t^++t^-$ and $t=t^+-t^-$.}  % follows fromthat by splitting up the summands into positive and negative parts,
\begin{lemma}
\label{lemma 1}
\begin{eqnarray}
\label{positive part} \dtv(X,Y) &=& \sum_k (\p(X=k)-\p(Y=k))^+ \\
\nonumber & = & \sum_k (\p(X=k)-\p(Y=k))^-.
\end{eqnarray}
\end{lemma}

%For example, 
For example, when $X$ is  a Bernoulli random variable with  parameter $\theta$,\footnote{so that $\p(X=1)=\theta= 1-\p(X=0)$} and $Y$ is Bernoulli with parameter $\theta'$, the total variation distance is  $|\theta-\theta'|$.

Since our sample space is discrete, and the labels of the socks have no intrinsic meaning, it does not make sense to consider metrics such as Wasserstein distance, which assigns a metric on the sample space.  A popular alternative is the Kullbach-Liebler divergence, or relative entropy, which has the undesirable property of being asymmetric.  While in many circumstances total variation distance is too strong, we find it here

\begin{defn}\label{define D}
Given a discrete probability distribution $\bp$, let $X  $ have the Method 1 distribution given by 
\eqref{def M1}, let $Y  $ have the Method 2 distribution given by \eqref{def M2}, and define the 
\emph{discrepancy} of $\bp$ by
\begin{equation}\label{def D}
\DD = \dtv(X(\bp),Y(\bp)).
\end{equation}
\end{defn}
We could have written $\DD=\dtv(X,Y)$ above, but we prefered $ \dtv(X(\bp),Y(\bp))$, to emphasize that $\DD$ is the total variation distance between two probability laws, with each law being a function of a third underlying law $\bp$.
\section{Special Cases}

\subsection{Dimension $n=1$:  two colors of socks}\label{sect 2 colors}

In the case $n=1$, we write $\bp = (p_0,p_1)=(x,1-x)$. %$p_1=p$ and $p_0=1-p$.  %, and by symmetry, it suffices to consider $1/2 \leq p \leq 1$.  %, where $1/2 \leq p\leq 1$.  
The discrepancy $\DD = \dtv(X,Y)$ simplifies, via Lemma \ref{lemma 1}, to $|d_1|$, where 
\begin{eqnarray*}
d_1(x) = \p(X=0)-\p(Y=0) & = & \frac{x^2}{x^2+(1-x)^2} - (x^2+2(1-x)x^2) .
%N_2  := \p(X=0)-\p(Y=0)& = & \frac{(1-p)^2}{p^2+(1-p)^2} - ((1-p)^2+2p(1-p)^2).
\end{eqnarray*}
%By Lemma \ref{lemma 1}, $\dtv(X,Y) = \left| d_1 \right|$.
The expression $|d_1(x)|$ is plotted in Figure \ref{fig p}.  %By symmetry, $d_1(x) = - d_1(1-x)>0$ for $0\leq x \leq \frac{1}{2}$.%, with $d_1(x)>0$.  %so it suffices to consider $d_1(x)$, which is rational.
%The maximum value of $\dtv$ is then the maximum value of $d_1$ over all $1/2\leq p \leq 1$.  

Since $d_1$ is a rational function in one variable, it is easily optimized over $x\in [0,1]$. %we can easily take the derivative and find the critical numbers.  
We outline our procedure as a preparation for the more difficult case in Section \ref{sect 3 colors}.  We first put the derivative %is put 
over a common denominator, which is strictly positive for $0\leq x \leq 1$, and focus our attention on the numerator.  The numerator is a sixth degree polynomial in $x$ of the form $4 \left(-x+7 x^2-18 x^3+24 x^4-18 x^5+6 x^6\right)$, having four real roots: $0$, $1$, 
\begin{equation}\label{pmax}
x_1 := \frac{1}{6} \left(3 + \sqrt{3 \left(-3+2 \sqrt{3}\right)}\right) \doteq 0.696660, % 0.69665994659516431958,
\end{equation}
and the conjugate, $1-x_1$. The list of roots already includes both endpoints of the domain $[0,1]$.  The cusp for $|d_1(x)|$ at $x=1/2$ is also critical, with $|d_1(1/2) = 0|$ corresponding to the uniform case.  
Evaluating $|d_1(x)|$ at these five critical numbers exhausts all possible extremes, and the maximum value is 
$d_1(x_1)= 
%Since $p=0$ and the `$-$' solution of Equation \ref{pmax} are outside of the domain, we need not consider $\dtv$ evaluated at these points.  After testing the other critical numbers and the endpoint $p=1/2$, we conclude that the maximum value of $\dtv$ is given by $d_1$ evaluated at the value of $p$ given by the `+' solution to Equation \ref{pmax}, which gives $\max_{p} \dtv = 
 \frac{1}{\sqrt{135+78 \sqrt{3}}} \doteq 0.0608468$.
% More decimals 0.060846799231813547761

\begin{figure}[ht]
\includegraphics{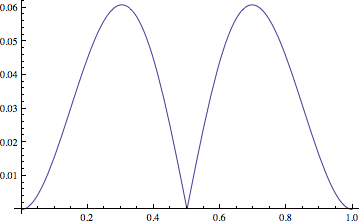}
\caption{Plot of $\DD$ for $\bp=(x,1-x)$, as a function of $x \in [0,1]$.\label{fig p}}
\end{figure}

%\clearpage

\subsection{Dimension $n=2$:  three colors of socks}\label{sect 3 colors}

\ignore{
\begin{figure}
\includegraphics[scale=0.6]{n2Plot1}
%\end{figure}
%\begin{figure}
%\label{fig ab2}
\includegraphics[scale=0.6]{n2Plot2}
\caption{Two different views of the plot of $\dtv$ over $\mathcal{R}$ in the case $n=2$.\label{fig ab}
}
\end{figure}
}

The case $n=2$ can be set up similarly to $n=1$, but now we have three cases of possible signs underlying absolute values.   Each case is a smooth, two-dimensional surface, and we find extremes by checking all critical values arising from points where the gradient vanishes, and on the boundary.  To avoid subscripts, we switch notation from $\bp=(p_0,p_1,p_2)$ to $\bp=(a,b,c)$, and define
$$
f(a,b,c):= a^2 (1 + 2(b + c) + 6bc),
$$ 
$$
T(a,b,c)= \frac{a^2}{a^2+b^2+c^2} - f(a,b,c),
$$
so that when $\bp=(a,b,c)$, with $a$ being the probability that a single sock has  color 0,
$T(a,b,c)=\p(X=0)-\p(Y=0)$.  Note that $T(a,b,c)=T(a,c,b)$.  
Exchanging the roles among colors 0, 1, 2,  we have $T(b,a,c)=\p(X=1)-\p(Y=1)$
 and $T(c,a,b)=\p(X=2)-\p(Y=2)$.
%Combining \eqref{def M1}, \eqref{def M2}, and \eqref{def dtv}, when socks have three possible colors, with probabilities given by $\bp=(a,b,c)$, the discrepancy between the two methods of selecting a pair is given by
From 
Definitions \ref{define dtv} and \ref{define D},
%Definition \ref{def D} and Lemma \ref{lemma 1}, 
when $\bp=(a,b,c)$,  
$$
2 \DD =| T(a,b,c)| + |T(b,a,c)|+|T(c,a,b)|.
$$
 
The expression above has the form $|T_1|+|T_2|+|T_3|$, and  the absolute value function is an obstacle to taking the gradient.  But by taking the eight cases for the sign, each of the expressions $\pm T_1 \pm T_2 \pm T_3$ is a %smooth
rational  function. %, that yields to calculus methods. 

A straightforward parameterization of the two-dimensional set of probabilities $(a,b,c)$ would have
$a \ge b \ge 1-a-b \ge 0$, implying that $T_1 \ge 0$ and $T_3 \le 0$, so that there are only two cases, according to  the sign of $T_2$.  A major obstacle to this approach is the boundary, which is complicated, so instead we parameterize in terms of $(x,y) \in [0,1]^2$ as follows:
$$  \bp(x,y) = (a,b,c) \ \mbox{where } 
t=1+x+y, a=\frac{1}{t},b=\frac{x}{t},c=\frac{y}{t}.$$   
Now taking $a=a(x,y)$ and so on, we have three functions defined on $[0,1]^2$,
\begin{eqnarray}
\nonumber T_1(x,y)  & := &  T(a,b,c), \\
\nonumber T_2(x,y)  & := &  T(b,a,c), \\
\nonumber T_3(x,y)  & := &  T(c,a,b).
\end{eqnarray}

The total variation distance is 
given by % one-half of
\begin{equation}\label{dtv n=2}
2\dtv(X,Y) = |T_1(x,y)|+|T_2(x,y)|+|T_3(x,y)|.
\end{equation}
Since $1 \ge x,y$, we have $a \ge b,c$ and 
since the largest mass is at 1, we know that for all $x,y\in [0,1]$, $T_1(x,y) \ge 0$. 

We can eliminate the case $T_1 \ge 0, T_2 \ge 0$ and $T_3 \ge 0$, as this implies $ T_1=T_2+T_3=0$ since $T_1 + T_2 + T_3=0$.   By Lemma \ref{lemma 1} this case gives $\DD=0$, not of interest in the search for the maximum value.
There are three remaining cases of sign to consider.  Let
\begin{eqnarray*}
d_1(x,y) & = & T_1(x,y) + T_2(x,y) -T_3(x,y), \\
d_2(x,y) & = & T_1(x,y) - T_2(x,y) +T_3(x,y), \\
d_3(x,y) & = & T_1(x,y) - T_2(x,y) -T_3(x,y). 
\end{eqnarray*}
%Then $\max \dtv(X,Y) = \max(d_1,d_2,d_3)$, and so it suffices to check the maximum values of each of these functions, which are not just rational functions in $x$ and $y$, but rational functions whose denominators are always strictly positive when $x,y\geq 0$.
Then $\max \dtv(X,Y) = \max(d_1,d_2,d_3)$, and so it suffices to check
the maximum values of each of these rational functions.

Let us consider $g(x,y) := d_1(x,y)$.\footnote{The term $d_2$  becomes $d_1$ under the interchange of $x$ and $y$, so no further work is required for $d_2$.  For $d_3$, the corresponding $h_x$ and $h_y$, after
cancellation of a common factor, have total degree 6 each, and one must account for the 36 solutions guaranteed by Bezout's Theorem.}
  Since $g$ is a rational
function in two variables, it is elementary to calculate the partial derivatives with
respect to $x$ and $y$, denoted $g_x$ and $g_y$, respectively.  What
is \emph{not} so elementary is finding all solutions $(x,y)$ to the
system $g_x(x,y) = g_y(x,y)=0$.  This set, $V(g_x,g_y) := \{(x,y): g_x=g_y=0\}$,
also known as the affine \emph{variety} defined by $g_x, g_y$, is what we
wish to find; a good introductory text on this subject is \cite{IVA}.

Continuing with this example, even though $g_x$ and $g_y$ are rational functions, when each is rationalized it is clear that for $x,y\geq 0$ the denominator is always positive, and hence plays no role in characterizing the set of points in the variety $V(g_x,g_y) \cap [0,1]^2$.  Thus we may simply find the variety of the numerators restricted to $[0,1]^2$, denoted $h_x$ and $h_y$, respectively, which are bivariate polynomials.

  A generalization to the Theorem of Algebra due to Bezout (see for example Chapter 5, Section 7 of \cite{IVA}) can be used to verify that all  solutions have been found\footnote{The precise form of the theorem requires several definitions and is not intended to be the focus; instead, we merely require assurance that the solutions found by Mathematica\textsuperscript{\textregistered} \cite{Mathematica} are exhaustive, since they are easily verified.%  Of course, when $h_x$ and $h_y$ share common factors there will be an infinite number of solutions, but once those are cancelled out we are left with a finite number of points.}.
  }.
  %For a given system of multivariate polynomials, Mathematica guarantees that it will find all solutions or give a warning saying otherwise.   
%   of total degree 7
%   $7 \times 7 = 49$
In this case, after dividing out by a common factor of $x$, %(which covers solutions of the form $x=0$, $y$ can be anything), 
the two polynomials each have total degree 7.  Bezout's theorem guarantees $7\times 7 = 49$ solutions total including multiplicities, but some of these are solutions ``at infinity."\footnote{Here is a simple analogy: How many times will a parabola intersect a line?  A parabola has degree 2 and a line has degree 1.  Suppose our parabola is $y=x^2$: then if our line is 1) $y=x-1$, then there will be no intersections; 2) $y=0$, then there is one intersection of multiplicity 2;  3) $y=x$, then there are two unique intersections of multiplicity 1 each; 4) $x=a$, for any real $a$, then there is one intersection of multiplicity 1.  By using an appropriate transformation into the projective plane, one can guarantee exactly two solutions in all cases.}  Mathematica\textsuperscript{\textregistered} finds a set of 19 unique, easily-verified solutions; when including multiplicities, this accounts for 39 of the total solutions.  
%The remaining 10 solutions are indeed ``at infinity," but we do not delve into this further and instead invite the reader to work out the detailed calculations.
%By hand, we can find 10 solutions at infinity, so all 49 solutions are accounted for.
By hand we can find 10 solutions at infinity, so all 49 solutions have been addressed. %are accounted for.

\ignore{
To find the maximum of each $d_i(x,y)$, $i=1,2,3$, we need to find $\{ (x,y)\in [0,1]: \nabla d_i(x,y)=0\}$.  In Mathematica, the command 
$$
{\tt \text{Solve}[\{D[\text{d1}[x,y],x],D[\text{d1}[x,y],y]\}==0,\{x,y\}]}
$$
will attempt to find all solutions to this set of equations, and if it cannot certify that it has found \emph{all} solutions it will produce a warning.  
%, which it does if the functions are defined as they are above.    
Rather than input the system above directly into Mathematica\textsuperscript{\textregistered}, it is better to first rationalize the expression by applying {\tt Together}, followed by {\tt Numerator}, which returns the numerator of the rationalized expression.  A final step is to compute the greatest common divisor, using {\tt PolynomialGCD}, check for its solutions, and then divide each equation by this expression to obtain an irreducible system\footnote{This is an important step as the GCD may have an infinite number of solutions which are more easily handled by hand, and may cause Mathematica\textsuperscript{\textregistered} to give a warning.  By dividing through by this common divisor, we are left with an irreducible set of equations with a provably finite number of roots, which will yield to standard techniques in the case of polynomials}.  A good introductory reference for the study of polynomials is (IDEALS, VARIETIES, and ALGORITHMS).
} % end ignore

%and then multiply through by the denominators, which are always strictly positive.  
%Once we have these simplified expressions, we can then do one final step by dividing through by the greatest common divisor of the two expressions.

% is the same as  used found the set of $(x,y)$ such that the solutions to gradient is 0.
%took derivatives with respect to $x$ and $y$ and solved for when the system was simultaneously 0.  

% dtv = 0.08429419234614604446250630128463933727501473480639931042529739419880\
% 4015476303480007322587184827953

%Among all solutions found, we are only interested in those that lie in $[0,1]^2$, leaving just a few solutions to check.  %Among the solutions found, 
We obtain the largest value %$\dtv$ 
of $\dtv$
%, approximately 0.0842942, 
from the point $(x,y)$ 
% \doteq (.35909, .35909)$.  The %Their 
%exact values
given by\footnote{The Mathematica\textsuperscript{\textregistered} expressions are
$$ x=\text{Root}\left[1+4 \text{$\#$1}-14 \text{$\#$1}^2-4 \text{$\#$1}^3-34 \text{$\#$1}^4+20 \text{$\#$1}^5\ \&,2\right],$$
$$ \dtv =\frac{1}2 \text{Root}[32000+168192 \text{$\#$1}-4557600 \text{$\#$1}^2+14567472 \text{$\#$1}^3$$
$$-821583 \text{$\#$1}^4+314928 \text{$\#$1}^5\&,2].$$
}
% are
\begin{align}
\nonumber x \in (0,1): & \ \ 1+4x-14x^2-4x^3-34x^4+20x^5 = 0, \\ % {\tt Root}[1 + 4\#1 - 14 \#1^2 - 4 \#1^3 - 34 \#1^4 + 20 \#1^5\ \&, 2] \\
\nonumber y: & \ \ y=x, \\
\label{DD val}
\begin{split}
  2\dtv =z \in (0,0.2) : & \ \ \ \  32000 + 168192 z \\ 
  &- 4557600 z^2 + 14567472 z^3 \\
  &- 821583 z^4 + 314928 z^5 = 0. %\dtv & = & {\tt Root}[32000 + 168192 \#1 - 4557600 \#1^2 + 14567472 \#1^3 \\ 
%       & &\ \ \ \ \ \ \   - 821583 \#1^4 + 314928 \#1^5\ \&, 2],
\end{split}
\end{align}
%Thus, our unique maximum value $\DD = 0.0842942$ is produced by $\bp = (0.582011, 0.208994, 0.208994)$, which is of the form
This solution is of the form
$$
\bp=\left(x_2,\frac{1-x_2}{2},\frac{1-x_2}{2} \right)
$$
for the value of $x_2\in [0.5,0.6]$ that solves $-5+42 x_2-114 x_2^2+168 x_2^3-153 x_2^4+54 x_2^5 =0$, with
\begin{equation}\label{def x2}
x_2 \doteq 0.582011, \ \ \ \DD \doteq 0.0842942;
\end{equation}
the exact value of $\DD$ given by Equation \eqref{DD val}.

\section{Conjectures about the largest possible discrepancy}\label{sect conj}

The weakest conjecture is that there is some nontrivial upper bound on discrepancy.  Formally,
we define the universal constant for the pair discrepancy by 
\begin{equation}\label{def universal}
  \ell_0 := \sup_\bp \DD,
\end{equation}
where the supremum is over all 
distributions $\bp$ on a finite or countable set of colors.  Since total variation distance is always less than or equal to 1, trivially $\ell_0 \le 1$, and the conjecture is

\begin{conjecture}\label{weakest conjecture}
The constant defined by \eqref{def universal} is \emph{strictly} less than 1, i.e.,
\begin{equation}\label{weak equation}
  \ell_0 < 1.
\end{equation}
\end{conjecture} 

\subsection{Conjectures for a finite number of colors}

If there are a finite number of colors, say $n+1$ with $n \ge 0$, then we can relabel the colors as $0,1,\ldots,n$ so that $\bp=(p_0,\ldots,p_n)$ with
\begin{equation}\label{sorted}
%\bp=(p_0,\ldots,p_n) \mbox{ with }
%p_0 \ge p_1 \ge \cdots \ge p_n >0,  \  p_0+p_1+\cdots+p_n=1.
p_0 \ge p_1 \ge \cdots \ge p_n  \ge 0,  \  p_0+p_1+\cdots+p_n=1.
\end{equation}
Given $n>0$, and $x \in [\frac{1}{n+1},1)$, let
\begin{equation}\label{one parameter}
\bpnx = \left(x,\frac{1-x}{n},\ldots,\frac{1-x}{n} \right),
\end{equation}
which, due to $x \in [\frac{1}{n+1},1)$, satisfies \eqref{sorted}.

With the notation \eqref{one parameter}, % this notation, 
the result of Section \ref{sect 3 colors} may be summarized as:  for $n=2$, over all probability distributions on $n+1$
 colors standardized to satisfy \eqref{sorted}, the maximum value of $\DD$ is achieved, uniquely, 
 at $\bp =% \bpnx$,
 \bp(2,x)$,  
 with $x =x_2$ as specified by \eqref{def x2}.

For each $n>0$, \eqref{one parameter} defines a \emph{one parameter family} of probability distributions.
%
%  NOTE TO authors:  yes, for EACH n there is a one parameter family;  hence we have families.
%
At the endpoint $x=1/(n+1)$, $\bpnx$ is a uniform distribution.  Now suppose that $x \in (1/(n+1),1)$, so that $\bpnx$ has $p_0 > p_1 = p_2 = \cdots = p_n>0$.  It is obvious from \eqref{def M1} that \mbox{$\p(X=0)>\p(X=1)=\cdots = \p(X=n)>0$,} and 
Lemma \ref{ratio lemma} implies 
%  WARNING:  check this claimed implication
that \mbox{$\p(Y=0)>$} $\p(Y=1)=\cdots = \p(Y=n)>0$.  That is, both $X$ and $Y$ have distributions in the same one parameter family.  Finally, \eqref{ratio XY} implies that $\p(X=0)>\p(Y=0)$, while for $i=1$ to $n$, $\p(X=i)<\p(Y=i)$, and hence using \eqref{positive part}, for each $n>0$ and $x \in (\frac{1}{n+1},1)$,  $\bp=\bp(n,x)$ has the simplified expression for its discrepancy,
\begin{align}
\nonumber \DD & = \p(X=0)-\p(Y=0) \\
 & = \frac{x^2}{x^2+\frac{(1-x)^2}{n}} - x^2 \sum_{k=0}^n (k+1)! \binom{n}{k} \left(\frac{1-x}{n}\right)^k. \label{was19}
% \nonumber \DD & = \p(X=0)-\p(Y=0) \\
% & = \frac{x^2}{x^2+\frac{(1-x)^2}{n^2}} - x^2 \sum_{k=0}^n (k+1)! \binom{n}{k} \left(\frac{1-x}{n}\right)^k.
%  OLD value, with the missing factor of n on the bottom of the first fraction --- since the FullSimplify result below
%  MAY HAVE used this incorrect expression
\end{align}
% ------------- Note Mathematica's FullSimplify reduces this to
% x^2 \frac{n^2 \left(1+\frac{e^{-\frac{n}{-1+x}} (-1+x+n x) \text{ExpIntegralE}\left[-1-n,-\frac{n}{-1+x}\right]}{-1+x}\right)}{(1+n) (-1+x)^2}
% -------------

%for $\bp$ in our one parameter family.
%More --- explicit simplification, as a function of n
%More, pretty plot, pretty please. 

\begin{figure}
\includegraphics[scale=0.6]{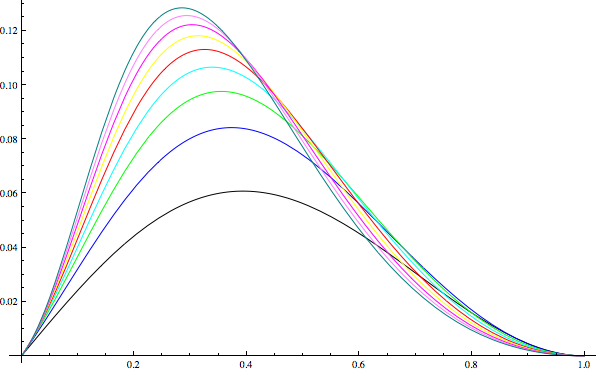}
%\end{figure}
%\begin{figure}
%\label{fig ab2}
 \caption{$\DD$ for the one parameter families \eqref{one parameter}, $n=1$ to 9.
 %, scaled so that $x_n\in [\frac{1}{n+1},1]$ is linearly transformed into $[0,1]$.}
For each $n$, we plot $\frac{n+1}n x - \frac{1}n$ versus $D(\bp(x,n))$, so that all 9 graphs have domain [0,1].
\label{fig one parameter}}
\end{figure}

\begin{conjecture}\label{conjecture}
For every nonnegative integer $n$, among all probability distributions on $n+1$ colors, the maximum value of $\DD$ is achieved by a distribution of the form $\bpnxn$.
\end{conjecture}  

A slightly stronger conjecture is the following:

\begin{conjecture}\label{stronger conjecture}
For every nonnegative integer $n$, among all probability distributions on $n+1$ colors, the maximum value of $\DD$ is achieved 
\emph{uniquely} by $\bpnxn$, where $x_n = {\rm argmax}_x \ D(\bpnx)$.
\end{conjecture}  

We cannot prove Conjecture \ref{conjecture}, but we believe it to be true, for the following reasons.
\begin{enumerate}
\item  It is true, trivially for $n=0$ and $n=1$, and by Section \ref{sect 3 colors}, for $n=2$. % It also agrees with the two-parameter families for the cases $n=3,\ldots,7$ in Section \ref{sect 4 colors}.

\item  By broad analogy, many symmetric payoff functions achieve their extreme values at points with lots of symmetry.  Indeed, Theorem \ref{uniform thm} asserts that for each $n$, $\DD$ achieves its \emph{minimum} value, zero, at the uniform distribution, corresponding to the maximum conceivable symmetry in $\bp$,  while the family in \eqref{one parameter}  corresponds to \emph{breaking} symmetry somewhat, but as little as possible.

\item  The one parameter family \eqref{one parameter} shows up in other extremal problems which share the feature that the \emph{labels} on the colors are irrelevant, and only the values of the probabilities matter.  In particular, in information theory, the one parameter families show that ``Fano's inequality is sharp;'' see Cover and Thomas \cite{cover}, (2.135) on page 40.

\item  %For moderate values of %$n$,  
%say $n=3$ to around $n=8$,  
%$n\leq 8$, 
For the moderate values $n=3,4,\ldots,8$, 
when generating %say 
a million random
%probability distributions 
points from the $n$-dimensional region specified by \eqref{sorted},  
the largest observed $\DD$ in the sample  came from a $\bp$ that was close, by eye, to the form of \eqref{one parameter}.

\end{enumerate}

\ignore{   %WARNING. perhaps we should rescue the cor.
\begin{conjecture}\label{conj inc}
The function $D(\bpnx)$ is \emph{strictly} increasing for $n \geq 0$.
\end{conjecture}
\begin{proof}
Not sure, corrected previous mistake.
%Trivial in light of the fact that $\frac{(n)_k}{n^k}$ is strictly increasing for $n\geq 0$.
\end{proof}

As a result, we get the following corollary.

\begin{cor}
If Conjectures \ref{conjecture} and \ref{conj inc} are true, then the maximum value of $\DD$ is not achieved by any $\bp$ with finite support.
\end{cor}
}  % end ignore --- recall the WARNING

The table below summarizes approximate 
extreme 
values under the one parameter families \eqref{one parameter} for $n=1, \ldots, 9$, 
using the notation $x_n = {\rm argmax}_x D(\bp(n,x))$.   $$
\begin{array}{ll}
%x_1 = 0.6966599, & D(x_1) = 0.06085 \\
% x_1 =      0.6966599465951643196
% D(x_1) = 0.060846799231813547761
x_1 =      0.6966599465951643196 &  D(x_1) = 0.06084679923181354776\\
%x_2 = 0.582011, & D(x_2) = 0.0842942 \\
% x_2 =      0.5820110139097399105
% D(x_2) = 0.084294192346146044463
x_2 =      0.5820110139097399105 & D(x_2) = 0.08429419234614604446 \\
%x_3 = 0.516003, & D(x_3) = 0.097663 \\
% x_3 =         0.5160030571683498864
% D(x_3) = 0.097662973595423267583
x_3 =         0.5160030571683498864 & D(x_3) = 0.09766297359542326758 \\
%x_4 = 0.47108, & D(x_4) = 0.106614 \\
% x_4 =       0.47108123676339401055
% D(x_4) = 0.10661363736945495196
x_4 =       0.4710812367633940106 & D(x_4) = 0.10661363736945495196 \\
%x_5 = 0.43766, & D(x_5) = 0.11316 \\
% x_5 =       0.43765985648455615135
% D(x_5) = 0.11316011048732238932
x_5 =       0.4376598564845561514 & D(x_5) = 0.11316011048732238932\\
%x_6 = 0.41138, & D(x_6) = 0.118225\\
% x_6 =      0.41138114794484457392
% D(x_6) = 0.11822473613430355437
x_6 =      0.4113811479448445739 &  D(x_6) = 0.11822473613430355437\\
% x_7 = 0.389926, & D(x_7) = 0.122298.
% x_7 =       0.38992587701011184638
% D(x_7) = 0.12229838762442936532
x_7 =       0.3899258770101118464 & D(x_7) = 0.12229838762442936532 \\
% x_8 = 0.37192393048779581350
% D(x_{8}) = 0.12566994796517442344
x_8 = 0.3719239304877958135 & D(x_{8}) = 0.12566994796517442344 \\
% x_{9} = 0.35650339133887214100
% D(x_{9}) = 0.12852218802677888163
x_{9} = 0.3565033913388721410 & D(x_{9}) = 0.12852218802677888163 \\
% x_{10} = 0.34307756601874660839
% D(x_{10}) = 0.13097724663540745829
%x_{10} = 0.3430775660187466084 & D(x_{10}) = 0.13097724663540745829.
\end{array}
$$
Figure \ref{fig one parameter} shows, for $n=1$ to 9,  $D(\bp(x,n))$ for $x \in [\frac{1}{n+1},1]$;  the graph plots
$\frac{n+1}n x - \frac{1}n$ 
%$1-\frac{n+1}n (1-x)$
versus $D(\bp(x,n))$, so that all 9 graphs use the same domain, [0,1].
%$D(x_n)$ for $x_n \in [\frac{1}{n+1},1]$, linearly transformed so that for each $n$, $x_n\in [0,1]$.

\section{Limit analysis of the one parameter family}
% end ignore soft limit discussion

 \begin{theorem}\label{ell theorem}
For $c \in (0,\infty)$ define
\begin{equation}\label{ell def}
\ell(c) = \frac{c^2}{1+c^2} -  \int_0^\infty c^2 t e^{-c t - t^2/2}\, dt  .
\end{equation}

For any $c \in (0,\infty)$ and $n > 1/c^2$, let $\psn = \bp(n,c/\sqrt{n})$ be  the distribution governed by \eqref{one parameter} with $x=c/\sqrt{n}$. 
Then
\begin{equation}
\label{ell limit}
 \lim_{n \to \infty} \DDn = \ell(c),
 \end{equation}
 where $\ell$ is defined by \eqref{ell def}.
\end{theorem}

%Our proof uses Poisson processes, and is postponed until the end of the section.  %Section \ref{sect pp}.  

%\begin{proof}[Proof of Theorem \ref{ell theorem}]
\begin{proof}
%\section{A Poisson process limit}
%\label{sect pp}

Extend Method 2 beyond the time of the first matching pair; i.e., pick socks forever.  For each color $i$ let $N_i$ be the number of sock picks needed to get the second sock of color $i$.  As the color varies, these random variables are \emph{dependent}, since for any two distinct colors $i,j$ and time $n \ge 2$, 
$0 = \p(N_i = N_{j}=n) < \p(N_i = n) \, \p(N_{j}=n)$.  There is a standard technique to deal with this dependence, used in Markov chains\footnote{see for example \cite{lawler}.},  which is to take a sequence of independent exponentially distributed holding times $Y_1,Y_2,\ldots$, with $\p(Y_n >t)=e^{-t}$, and declare that the $n$th sock arrives at time $Y_1+Y_2+\cdots+Y_n$.\footnote{The number of socks picked by time $t$ is thus Poisson distributed, with mean $t$.  Write 
 $C_i(t) = $ the number of socks of color $i$ chosen by time $t$.  As $i$ varies, the counts $C_i(t)$ are mutually independent;  this observation  is known as \emph{Poissonization}. See exercise XII.6.3 in Feller
 \cite{Feller1}.}  With values in $(0,\infty)$, the
time $T_i$ at which color $i$ is first seen for the second time can be expressed as $T_i = Y_1+\cdots + Y_{N_i}$.
The distribution of the color of the first matching pair found, initially specified by \eqref{def M2}, can also be expressed as 
$$
\p(Y=i) = P(T_i < \min_{j\neq i} T_j).
$$
For each color $i$, the times at which socks of color $i$ arrive form a Poisson arrivals process with rate $p_i$, and
as the color varies, these processes are mutually independent;  in particular  the second arrival times $T_i$
are mutually independent.

We are considering socks distributed according to $\bp(n,c/\sqrt{n})$, that is, with $y := (1-c/\sqrt{n})$,
\begin{equation}\label{n distribution}
p_0 = c/\sqrt{n}, p_1=y/n,p_2=y/n,\ldots,p_n=y/n.
\end{equation} 
Speed up time by a factor of $\sqrt{n}$;  now socks of color 0 arrive at rate $c$, and for each other color $i=1$ to $n$, socks of color $i$ arrive at rate $p_i \sqrt{n} = y/\sqrt{n}$.  For $t>0$, and for each $i=1$ to $n$, the number $Z$ of socks of color $i$ collected by time $t$ is Poisson with parameter $\lambda=ty/\sqrt{n}$, and the event $\{T_i >t\}$ is the event $\{Z<2\} = \{Z=0$ or 1\}, with probability
\begin{eqnarray}
\nonumber \p(T_i > t) &=&  \p(Z=0)+\p(Z=1) \\ 
\nonumber                   &=& e^{-\lambda}(1+\lambda)  \\
\label{fn} &=& \exp\left(-\frac{t y}{\sqrt{n}}\right)\left(1 + \frac{t y}{\sqrt{n}}\right) \\
\nonumber                   &=&  1 - \frac{t^2 y^2}{2n} + O(n^{-3/2}).
\end{eqnarray}
The easy way to see the result above is to argue that $\lambda$ is small, so $e^{-\lambda}(1+\lambda) =
(1 - \lambda + \lambda^2/2 -\lambda^3/6 + \cdots)(1+ \lambda) = 1 - \lambda^2 + \lambda^2/2 +O(\lambda^3)= 1 - \lambda^2/2 +O(\lambda^3)$.

The event $\{\min(T_1,\ldots,T_n)>t\}$ is the intersection of the events  \mbox{$\{T_i>t\}$}, so using the mutual independence, together with $y \to 1$, for each $t>0$,
\begin{eqnarray*}
\p(\min(T_1,\ldots,T_n)>t) &=& \p(T_1>t)^n = \left(1 - \frac{t^2 y^2}{2n} + O(n^{-3/2})\right)^n \\
&\to& \exp(-t^2/2).
\end{eqnarray*} 

Finally, we argue that the density of $T_0$, the second arrival time in a Poisson process with rate $c$, is given by
$$
f(t) = c^2 t e^{-c t}.
$$
This is a standard fact, known to some as the density of the Gamma distribution with shape parameter 2 and scale parameter $c$.  Using the independence of $T_0$ and $\min(T_1,\ldots,T_n)$,  we can condition on the value $t$ for $T_0$ to get 
\begin{eqnarray}
\nonumber \p_n(Y=0) &=& \p(\min(T_1,\ldots,T_n)>T_0) \\ 
\nonumber                    &=&  \int_0^\infty \p(\min(T_1,\ldots,T_n)>t) f(t) \ dt \\
\nonumber                    &=& \int_0^\infty \p(\min(T_1,\ldots,T_n)>t) \ c^2 t e^{-c t} \ dt \\
\label{unif conv}           &\to &  \int_0^\infty c^2 t e^{-c t} e^{-t^2/2}dt.
\end{eqnarray}

The above amounts to a calculation of the limit, as $n \to \infty$, of   $\p_n(Y=0)$, corresponding to Method 2
when the underlying colors come from \eqref{n distribution}.  

% ADD material on uniform convergence.
Of course, we must justify the passage to the limit in \eqref{unif conv}.  Here we have $f_n(t) := \p( \min(T_1, \ldots, T_n) ) \to \exp(-t^2/2) =: f(t)$ point-wise, for each $t>0$, but we claim in \eqref{unif conv} that the integrals also converge.  If we interpret the (improper) integral as the Lebesgue integral, then we can invoke the Monotone Limit Theorem: it is easy to check that $f_n(t) \geq f_{n+1}(t)\geq 0$ for all $n$, and that $\int f_1(t) dt < \infty$, hence $\int f_n(t) \, dt \to \int f(t)\, dt$.\footnote{See Chapter 2, Exercise 15 of Folland \cite{Folland}.}

Interpreting the improper integral as a Riemann integral requires more work to justify passage to the limit in \eqref{unif conv}, and is left as an exercise; see for example Chapter 7, Exercise 12 of Rudin \cite{Rudin}. 

% End of editing
For Method 1 the calculation is easier: using 
\eqref{def f2} we have $f_2 = p_0^2 + p_1^2 + \cdots + p_n^2 = (c/\sqrt{n})^2 + n(y/n)^2 = c^2/n + y^2/n$ and
$$
\p_n(X=0) = \frac{p_0^2}{f_2} = \frac{c^2/n}{c^2/n+y^2/n} = \frac{c^2}{c^2+y^2} \to \frac{c^2}{c^2+1}.
$$
 
At \eqref{was19} we had already argued that once $n$ is large enough that $p_0>p_1$ we have the simplification, for our one parameter family, that $\DDn = \p_n(X=0) - \p_n(Y=0)$.
 \ignore{
As in the argument used to prove Lemma \ref{lemma 1} for the first equality below, and using 
\eqref{equal} for the second equality, we have
\begin{eqnarray*}
   \p_n(X=0) - \p_n(Y=0) &=& \sum_1^n \p_n(Y=i) - \p_n(X=i) \\
   &=& n(\p_n(Y=1) - \p_n(X=1)).
\end{eqnarray*} 
Then again from Lemma \ref{lemma 1},
$$ \DDn = \dtv(X,Y) = |\p_n(X=0) - \p_n(Y=0) |.
$$
Using \eqref{ratio nonstrict}, once $n$ is large enough that $p_0>p_1$, we have $\p_n(X=0)/\p_n(Y=0) \ge 
\p_n(X=1)/\p_n(Y=1)$.  Hence $\p_n(X=0)/\p_n(Y=0) \ge 
(n\p_n(X=1))/(n\p_n(Y=1)) = (1-\p_n(X=0))/(1-\p_n(Y=0))$, hence $\p_n(X=0)\ge \p_n(Y=0)$, so the absolute values in the display above are superfluous.
} %end ignore
Combining this calculation of $\DDn$ with the limit values derived for $\p_n(Y=0)$ and \mbox{$\p_n(X=0)$}, \eqref{ell limit} follows.
\end{proof}

We note that instead of invoking Poissonization, as in the above proof, one can argue directly with the explicit expression in \eqref{was19}, to show that under $x=c/\sqrt{n}$ and $k = t \sqrt{n}$, the sum in \eqref{was19} is a Riemann approximation
for  $ \int_0^\infty c^2 t e^{-c t} e^{-t^2/2}dt$.
%;  the choice of procedure is purely a matter of taste.

\section{Discussion}

\begin{figure}[ht]
\includegraphics[scale=.5]{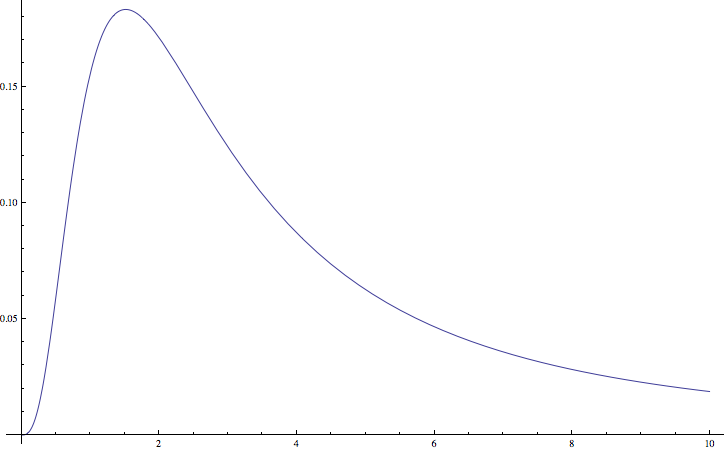}
\caption{Plot of $c$ versus $\ell(c)$ for $c=0$ to 10.  The maximum occurs at $c_0 \doteq1.514$ and has value
$\ell(c_0)\doteq  0.18320.$\label{fig ell}
%0.1832000624087106060398222190233627698786.
}
\end{figure}
%c  =1.513994072132400395081901794814404204764

If Conjecture \ref{conjecture} is true, it will follow that Conjecture \ref{weakest conjecture} is also true,
with the value of the universal constant for a pair of socks given by
\begin{equation}\label{universal constant}
 \ell_0 = \sup_c \ell(c) =0.1832000624087106\ldots  \ .
\end{equation}
The argument requires two parts.  The first part is to show that $\ell_0$, defined in \eqref{def universal} as the sup
of $\DD$ over \emph{all discrete} distributions, is equal to the sup over distributions with \emph{finite} support.  This is ``soft'' analysis, showing first that $\bp \mapsto \DD$ is continuous, hence given 
$\bp$ with discrepancy greater than $\ell_0-\varepsilon$ we can find a nearby distribution $\bp'$ with  finite support, close enough to $\bp$ to guarantee that its discrepancy is greater than $\ell_0 - 2 \varepsilon$.
The second part, giving the concrete value for $\ell_0$, 
%of proving that Conjecure \ref{conjecture} implies Conjecture \ref{weakest conjecture} with $ \ell_0 = \sup_c \ell(c)$ 
uses compactness:  given distributions $\psn = \bpnxn$
with discrepancies converging to $\ell_0$, the values $c_n := x_n \sqrt{n} \in [0,\infty]$, $n\geq 1$, lie in a %convex (my typo)
compact
set, and hence there must be convergent subsequences.  If $c_{n_k} \to c_0$ and $c_0 \in (0,\infty)$, then the proof of
Theorem \ref{ell theorem} already shows that the associated discrepancies converge to $\ell(c_0)$.  If 
$c_{n_k} \to c_0$ with $c_0=0$ or $c_0=\infty$, a small extension of the proof of Theorem \ref{ell theorem} would show that the associated discrepancies would converge to 0.  So indeed, $c_n \to c_0$ and $\DDn \to \ell(c_0)$.

\section{Shoes instead of socks: a matching left-right pair}

Suppose, instead of wanting to collect a pair of matching socks, we want a pair of matching shoes.
Naturally, this means one left shoe, and one right shoe,  both of the same color.  There are two reasonable ways to extend our study to this situation.

\subsection{One distribution for left colors, another distribution for right colors}

The setup here involves two 
discrete
probability distributions, say $\bp$ for the color $S$ of a left shoe, and $\bq$ for the color $S'$ of a right shoe.  The analog of \eqref{def f2}  is
\begin{equation}\label{def f2 LR}
f_2 = \p(S=S')  = \sum_i \p(S=S'=i) = \sum_{i} p_i \, q_i
\end{equation}
for the probability that a random left shoe and a random right shoe match.
We require that for at least one value $i$, $p_i q_i>0$.
The analog of \eqref{def M1} is the Method 1 distribution for the color $X=X(\bp,\bq)$ of a matching left-right pair
\begin{equation}\label{def M1 LR}
  \p(X=i) = \p(S=i|S=S') = \frac{p_i q_i}{f_2}.
\end{equation}

For method 2, we assume that at times $1,3,5,\ldots$, one left shoe is collected, and at times $2,4,6,\ldots$, one right shoe is collected. 
Suppose that at time $k-1$, there is not yet a matching left-right pair,  but at time $k$, there is;  then $Y = Y(\bp,\bq)$ is the color of the shoe collected at time $k$.\footnote{There are other sensible ways to determine the matching color under sequential collection of shoes, for example, selecting one left and one right shoe each at time $1, 2, 3, \ldots$ and breaking ties via a coin flip.  Even here, choices remain.  For example, if the outcome is $L_1=$ red, $R_1=$ blue, $L_2=$ red, $R_2=$ white, $L_3=$ white, $R_3=$ red,  then the tiebreak might be specified as equal odds for white versus red, or, since the available matches at time 3 
are $(L_1,R_3), (L_2,R_3)$, and $(L_3,R_2)$, as 2 to 1 in favor of red over white.  For this outcome, our specification in the the main text is white, since the earliest match occurs at time 5, when $L_3 =$ white is observed. } 

The analog of discrepancy is now

\begin{equation}\label{def D LR}
\DDLR = \dtv(X(\bp,\bq),Y(\bp,\bq)).
\end{equation}

It is fairly easy to see that for this situation, the analog of Conjecture \ref{weakest conjecture} is \emph{false}; that is, the supremum of the discrepancy over all pairs of distributions is no smaller than the trivial upper bound on total variation distance:
\begin{equation}\label{non conjecture}
1= \sup_{\bp,\bq} \DDLR .
\end{equation}

We give a brief sketch
of a proof of \eqref{non conjecture}: 
with $a=a(n)=n^{-1/4}$ and $b=b(n)=n^{-2/3}$  let $\bp=\bp(n, a)$ and $\bq=\bp(n,b)$; in other words, 
$p_0 = \p(S=0)=a$, $q_0 = \p(S'=0)=b$ and  for $i=1$ to $n$, $p_i =\p(S=i)=(1-a)/n$, $q_i = \p(S'=i)=(1-b)/n$, with $a=n^{-1/4}$, $b=n^{-2/3}$.  We have $p_0 q_0 = n^{-11/12}$ and 
$$ \sum_{i=1}^n p_i q_i  = n \frac{1-a}{n} \, \frac{1-b}{n} \sim \frac{1}n = o( p_0 q_0),
$$ 
so the Method 1 distribution converges to point mass at color 0, i.e., $\p_n(X=0) \to 1$.  To see that the Method 2 distribution has, in the limit, probability zero of getting color 0, consider collecting alternately left and right shoes forever.  At time $m= 2 n^{5/8}$,  we will have collected $n^{5/8}$ left and $n^{5/8}$ right shoes.  Thanks to the small value $q_0 = b = n^{-2/3}$, we expect only $n^{-1/24}$ left shoes of color 0 at time $m$, so with high probability, we do not yet have a matching pair of color 0.  But, at time $m$,  for \emph{each} color $i=1$ to $n$, the number of left shoes of color $i$ is Binomial($m,(1-a)/n$), and hence is greater than zero with probability asymptotic to $m/n \sim n^{-3/8}$.  Independently, the number of right shoes of color $i$ is greater than zero  with probability asymptotic to 
$ n^{-3/8}$;  hence the probability of at least one pair of color $i$ is asymptotic to $n^{-3/4}$.  The number $W$ of colors $i>0$ for which we have a pair has  $\e W \sim n^{1/4}$, and the $n$ events are negatively correlated with each other, so $\var W < \e W$.  By Chebyshev's inequality, $\p(W=0)  \le \var W / (\e W)^2 = O(n^{-1/4})$.
So at time $m$, we are unlikely to have any pair of color 0, and unlikely not to have at least one pair of some other color,  hence $\p_n(Y=0) \to 0$.

\subsection{With the constraint $\bp = \bq$}

Now suppose that we declare that the distribution $\bp$ for left shoes  and the distribution $\bq$ for right shoes must be equal.  This does not reduce consideration of the distribution of a matching pair to the situation for socks; 
under the alternating left-right procedure, if we get a blue left shoe at time 1, a red right shoe at time 2, and another blue left shoe at time 3, then we still have not collected a matching pair.

The analog of Conjecture \ref{weakest conjecture}, for the situation of a matching left-right pair of shoes under the constraint of equal distributions, is plausible:
\begin{conjecture}\label{weakest conjecture shoes}

\begin{equation}
\label{weak equation shoes}
  \sup_\bp \DDLL < 1.
\end{equation}
\end{conjecture}

Furthermore, we can even propose a value for the universal constant for shoes, given by the left side of
\eqref{weak equation shoes}.  It comes from an analog of Theorem \ref{ell theorem}.
This analog of Theorem \ref{ell theorem} is easiest to understand without the constraint $\bp=\bq$.

 \begin{theorem}\label{ell theorem shoes}
For $a,b \in (0,\infty)$ define
\begin{equation}\label{ell def2}
\ell(a,b) = \frac{ab}{1+ab} -  \int_0^\infty \left( ae^{-a t} + be^{-bt} - (a+b)e^{-(a+b)t} \right) e^{- t^2}\, dt  .
\end{equation}

For $a,b>0$ and sufficiently large $n$, let 
\begin{equation}\label{thm 3 setup}
\psn = \bp(n,a/\sqrt{n}), \ \qsn = \bq(n,b/\sqrt{n})
\end{equation} 
as in \eqref{one parameter}. 
Then
\begin{equation}
\label{ell limit shoes}
 \lim_{n \to \infty} D(\psn,\qsn) = \ell(a,b).
\end{equation}

\end{theorem}

\begin{proof}
The argument 
 closely follows the proof for Theorem \ref{ell theorem}.  We omit details, apart from sketching the main differences: under the distributions in \eqref{thm 3 setup},
collecting left-right pairs with mean $1/\sqrt{n}$ holding times between pairs, the left shoes of color 0 form a rate $a$ Poisson process, the right shoes of color 0 form a rate $b$
Poisson process;  $\p($no left 0 by time $t ) =e^{-at}$, $\p($no right 0 by time $t ) =e^{-bt}$,  and \emph{in the limit}, the two processes are independent, so $\p($no left 0 and no right 0 by time $t ) =e^{-(a+b)t}$.  Inclusion-exclusion and differentiation leads to the  limit density of the time $T_0$ at which
a left-right pair of color 0 is found,  $f(t) = \left( ae^{-a t} + be^{-bt} - (a+b)e^{-(a+b)t} \right)$, instead of the $c^2 t e^{-ct}$ of Theorem \ref{ell theorem}.   At time $t$, for each of the $n$ other colors we expect, asymptotically, $t/\sqrt{n}$ instances on the left, and
$t/\sqrt{n}$ on the right, with $t^2/n$ for the asymptotic chance of having a pair.   This leads to $\p(\min(T_1,\ldots,T_n)>t) \to \exp(-t^2)$, instead of the $\exp(-t^2/2)$ of Theorem \ref{ell theorem}. 
\end{proof}

\begin{figure}[h]
\includegraphics[scale=.7]{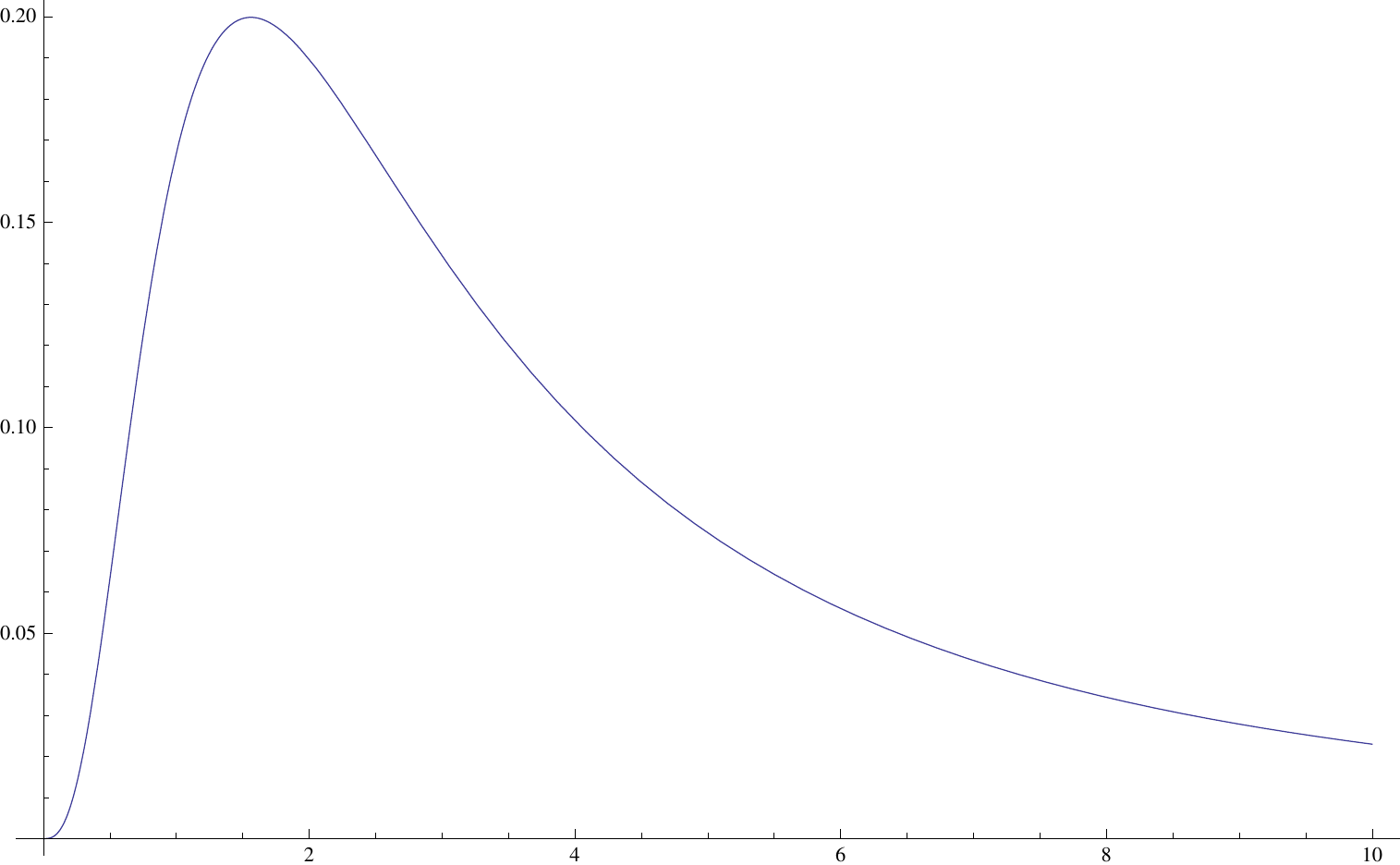}
\caption{Plot of  $\ell(a,a)$, the limit discrepancy $\DDLR$  when $\bp = \bq = \bp(n,a/\sqrt{n})$.  
%The maximum value $0.19980867405313229408$ occurs at $a = 1.5622394408366036099$}
The maximum value $0.19980867\ldots$ occurs at $a = 1.562239\ldots$.\label{ABcurve}}
\end{figure}

\begin{figure}[h]
\includegraphics[scale=.7]{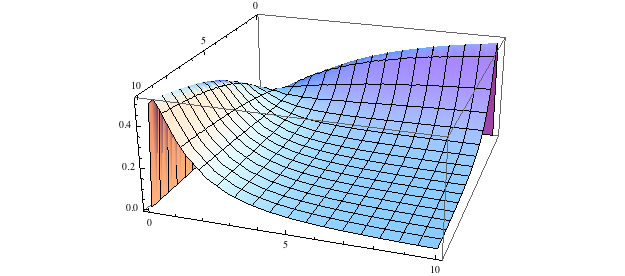}
\caption{Plot of $\ell(a,b)$, the limit discrepancy $\DDLR$  when $\bp = \bp(n,a/\sqrt{n})$ and $\bq = \bp(n,b/\sqrt{n})$.  The curve in Figure \ref{ABcurve}  %  why do we need to hardwire??  \ref{ABcurve} 
lies along the diagonal, splitting the plot into two symmetric pieces.\label{AB3Dcurve}
}
\end{figure}

While we do not have evidence for the analog of Conjecture \ref{conjecture}  --- indeed, it seems daunting to deal with the analog of Section \ref{sect 3 colors}, for left-right pairs under equal distribution for left and right ---  the analog of Conjecture \ref{weakest conjecture} \emph{combined with} \eqref{universal constant} is the following plausible conjecture.  See Figure  \ref{ABcurve}  for the source of the constant $.1998\ldots$ .

\begin{conjecture}\label{conjecture shoes}
 $$ \sup_\bp \DDLL = \max_{a} \ell(a,a) \doteq 0.199808674053.
 $$ 
 \end{conjecture}
 
\ignore{
 \section{Acknowledgements}

The second author would like to thank Steven Idhe for helpful information regarding the software package Bertini.  
}

 \bibliographystyle{plain}  
\bibliography{Socks}

\end{document}